\documentclass{birkjour}

\usepackage[utf8]{inputenc}
\usepackage{amsmath}
\usepackage[english]{babel}
\usepackage{bbm}
\usepackage{graphicx}
\usepackage[colorinlistoftodos]{todonotes}
\usepackage[colorlinks=true, allcolors=blue]{hyperref}
\usepackage{amssymb}
\usepackage{amssymb}
\usepackage{amsthm}
\usepackage{amsfonts}
\usepackage{amssymb}
\usepackage{amscd}
\usepackage{latexsym}
\usepackage{breqn}
 \newtheorem{thm}{Theorem}[section]
 \newtheorem{cor}[thm]{Corollary}
 \newtheorem{lem}[thm]{Lemma}
 \newtheorem{prop}[thm]{Proposition}
 \theoremstyle{definition}
 \newtheorem{defn}[thm]{Definition}
 \theoremstyle{remark}
 \newtheorem{rem}[thm]{Remark}
 \newtheorem*{ex}{Example}
 \numberwithin{equation}{section}

\begin{document}

%
%
%
%
%
%
%
%
%

\title[MTTOs: unbounded symbols, kernels and equivalence after extension]
 {Matrix-valued truncated Toeplitz operators: unbounded symbols, kernels and equivalence after extension}

\author[R. O'Loughlin]{Ryan O'Loughlin}
\address{
School of Mathematics\\
University of Leeds\\
Leeds\\
LS2 9JT\\
U.K.}

\email{matro@leeds.ac.uk}



\begin{abstract}
This paper studies matrix-valued truncated Toeplitz operators, which are a vectorial generalisation of truncated Toeplitz operators. It is demonstrated that, although there exist matrix-valued truncated Toeplitz operators without a  matrix symbol in $L^p$ for any $p \in (2, \infty ]$, there is a wide class of matrix-valued truncated Toeplitz operators which possess a matrix symbol in $L^p$ for some $p \in (2, \infty ]$. In the case when the matrix-valued truncated Toeplitz operator has a symbol in $L^p$ for some $p \in (2, \infty ]$, an approach is developed which bypasses some of the technical difficulties which arise when dealing with problems concerning matrix-valued truncated Toeplitz operators with unbounded symbols. Using this new approach, two new notable results are obtained. The kernel of the matrix-valued truncated Toeplitz operator is expressed as an isometric image of an $S^*$-invariant subspace. Also, a Toeplitz operator is constructed which is equivalent after extension to the matrix-valued truncated Toeplitz operator. In a different yet overlapping vein, it is also shown that multidimensional analogues of the truncated Wiener-Hopf operators are unitarily equivalent to certain matrix-valued truncated Toeplitz operators.
\end{abstract}

\maketitle

\section{Introduction}
The purpose of this paper is to study the matrix-valued truncated Toeplitz operator (abbreviated to MTTO). The MTTO is a vectorial generalisation of the truncated Toeplitz operator. We make a powerful observation, that when studying a given property of a MTTO it is often convenient to initially modify the MTTO by changing its codomain (in a natural way), then one can deduce results about the MTTO from the modified MTTO. This approach allows us to tackle problems which were previously out of reach concerning MTTOs with unbounded symbols. In particular for a MTTO which has a matrix symbol with each entry lying in $L^p$ for $p \in (2, \infty ]$, we describe the kernel of the MTTO as an isometric image of a $S^*$-invariant subspace and we also find a new form of Toeplitz operator which is equivalent after extension to the MTTO. We emphasise that although the results in this paper are for MTTOs, all the results are new even in the scalar case of the truncated Toeplitz operator. As the study of MTTOs is a recent endeavour, we devote the final section of this paper to discuss some applications of MTTOs to integral equations.

The spaces $H^p$ and $L^p$ will be defined on the unit circle, $\mathbb{T}$, where $1\leqslant p \leqslant \infty$. We write $(H^p)^n$ (respectively $(L^p)^n)$ to mean the column vector of length $n$ with each entry taking values in $H^p$ (respectively $L^p$). Background theory on the classical Hardy space $H^p$ can be found in \cite{duren1970theory, nikolski2002operators}. For $1<p \leqslant \infty$, we denote $L^{( p, n \times n)}$ to be the space of $n$-by-$n$ matrices with each entry taking values in $L^p$. We make an analogous definition for $H^{( p, n \times n)}$. For a matrix $M \in L^{( \infty, n \times n)}$ the adjoint of $M \in L^{( \infty, n \times n)}$ is denoted $M^*$. A $n$-by-$n$ matrix inner function $\Theta$ is an element of $H^{( \infty, n \times n)}$ such that for almost every $z \in \mathbb{T}$, we have $\Theta (z)$ is a unitary matrix. Throughout we use $\Theta$ to denote an $n$-by-$n$ inner function.

For $f \in (H^2)^n$, the backward shift, $S^*$, is defined by $S^*(f) = \frac{f - f(0)}{z}$. We know from the Beurling-Lax Theorem that $\Theta (H^2)^n$ is a shift invariant subspace. Therefore using orthogonality one can see that the (matricial) model space, $K_{\Theta} := \Theta\overline{(H^2_0)^n} \cap (H^2)^n$ is $S^*$-invariant. If we define $$(\overline{H^2_0})^n := \{ \overline{f} : f \in (H^2)^n \text{ and } f(0) \text{ is the zero vector} \},$$ then we have the orthogonal decompositions
$$
(H^2)^n = K_{\Theta} \oplus \Theta (H^2)^n,
$$
and $$
(L^2)^n = (\overline{H^2_0})^n \oplus (H^2)^n.
$$
This in turn gives us orthogonal projections $P_+: (L^2)^n \to (H^2)^n$, $P_{\Theta}: (L^2)^n \to K_{\Theta}$ and $Q_{\Theta} : (L^2)^n \to \Theta (H^2)^n$. For $k > n$ we write $P_n : (H^2)^{k} \to (H^2)^{n} $ to mean the projection onto the first $n$ coordinates.

Matrix-valued truncated Toeplitz operators were first defined in \cite{khan2018matrix} as a natural generalisation of truncated Toeplitz operators. They have further been studied in \cite{khan2020generalized, khan2020note}. We define the MTTO as follows. Let $G \in L^{( 2, n \times n)}$, consider the map 
\begin{equation}\label{MTTO}
    f \mapsto P_{\Theta}(Gf),
\end{equation}
defined on $K_{\Theta} \cap (H^{\infty})^n$. It is shown in Section 4 of \cite{khan2018matrix} that $K_{\Theta} \cap (H^{\infty})^n$ is dense in $K_{\Theta}$, so in the case when \eqref{MTTO} is bounded this uniquely defines an operator $K_{\Theta} \to K_{\Theta}$, which we denote $A_G^{\Theta}$ and call a \textit{matrix-valued truncated Toeplitz operator} (MTTO). We note that with this definition, all MTTOs are implicitly bounded. We call $G$ the symbol of the MTTO, and we note that if we have the additional assumption that $G \in L^{( \infty, n \times n)}$ then \eqref{MTTO} can always be extended to a bounded operator. In the case when $n=1$, we recover the well known bounded truncated Toeplitz operator.

We say $\Theta$ is pure if $||\Theta(0)|| < 1$. Matrix valued truncated Toeplitz operators with a pure inner function appear naturally in the Sz.-Nagy and Foias model theory for Hilbert space contractions. In particular, every bounded linear operator between two Hilbert spaces $T: H_1 \to H_2$ with defect indices $( n , n )$ and with the property that for all $h \in H_1$, $T^{*n} (h) \to 0$ (S.O.T) is unitarily equivalent to $A_z^{\Theta}$ for some $n$-by-$n$ inner function $\Theta$. See Section 2, page 33, of \cite{recentadvances} for a more detailed discussion. Although this is one of the main motivations for interest in the truncated Toeplitz operator (which is relevant when the defect indices are $(1,1)$), there has been very little research done in to the general case of the MTTO.

MTTOs also appear naturally in \cite{multibandspace}, where a dual band Toeplitz operator (defined in \cite{multibandspace}) is unitarily equivalent to a MTTO with a diagonal inner function $\Theta$. Dual band Toeplitz operators have applications in speech processing and signal transmission. We refer the reader to \cite{multibandspace} and further references thereafter for a detailed discussion.

Let $\theta \in H^2$ be a scalar inner function and let $\phi \in H^{\infty}$. We denote the Hankel operator with symbol $g \in L^{\infty}$, by $H_g : H^2 \to \overline{H^2_0}$. This is defined by $H_\psi(p) = P_{-}(\psi p)$, where $P_{-}:L^2 \to \overline{H^2_0}$ is the orthogonal projection. It is well known that many questions about Hankel operators can be phrased in terms of truncated Toeplitz operators with an analytic symbol. In particular the relation
$$
A_{\phi}^{\theta} = \theta H_{\overline{\theta}\phi} |_{K_{\theta}}
$$
has long been exploited. Making natural generalisations so that $\Psi \in H^{( \infty, n \times n)} $ and $H: (H^2)^n \to (\overline{H^2_0})^n$ is a Hankel operator on the vector-valued Hardy space, we can also write the relation
$$
A_{\Psi}^{\Theta} = \Theta H_{{\Theta}^*{\Psi}} |_{K_{\Theta}}.
$$
So, just as is true in the scalar case, the matricial Hankel operator and MTTO are fundamentally linked. This has applications in minimisation problems and Nehari's Theorem, see Section 2.2 of \cite{potapov}.

MTTOs also have a link to complex symmetric operators. Direct sums of truncated Toeplitz operators seem to play a role in some sort of model theory for complex symmetric operators (see Section 9 of \cite{recentprogressGarcia}). However, direct sums of truncated Toeplitz operators are in fact special cases of MTTOs with diagonal symbols and diagonal inner functions. So instead of considering what role direct sums of truncated Toeplitz operators play in the theory of complex symmetric operators, it may be more natural to consider what role MTTOs play in the theory of complex symmetric operators.

In Section \ref{sec:preliminaryresults} we make some key observations which allow us to define the modified MTTO. The modified MTTO turns out to be a crucial tool in later sections, particularly when we are trying to understand properties of MTTOs which do not possess a bounded symbol.

The kernel of a Toeplitz operator is easily checked to be nearly $S^*$-invariant, and this has long been exploited to study the kernels of vector-valued and scalar-valued Toeplitz operators \cite{hitt1988invariant} \cite{MR2651921}. The kernel of a (scalar) truncated Toeplitz operator with a bounded symbol was shown to be nearly $S^*$-invariant with defect 1 in \cite{oloughlin2020nearly}. Then consequently, this property was used to give a decomposition theorem for the kernel of a truncated Toeplitz operator with a bounded symbol. In Section \ref{sec:3.1} we use the modified MTTO to expand on previous studies to include the case where the symbol of the operator is not bounded. In particular, we decompose the kernel of a MTTO which possesses a symbol in $ L^{(p, n \times n )}$ for $p \in ( 2 , \infty)$ in to an isometric image of an $S^*$-invariant subspace.

In Section \ref{sec:4} we use the modified MTTO as a transitional device, which allows us to find an operator which is equivalent after extension to a MTTO which has a symbol in $ L^{(p, n \times n )}$ for $p \in ( 2 , \infty)$. Specifically, we first find a Toeplitz operator which is equivalent after extension to the modified MTTO, and then we change the codomain of this Toeplitz operator (as we have done when defining the modified MTTO) to produce an operator which is equivalent after extension to a MTTO which has a symbol in $ L^{(p, n \times n )}$ for $p \in ( 2 , \infty)$. This is a generalisation of the results in Section 6 of \cite{MR3634513}, where the authors construct a Toeplitz operator which is equivalent after extension to a truncated Toeplitz operator with a symbol in $L^{\infty}$.

In Section \ref{integralequation} we show there is a unitary equivalence between certain MTTOs and matricial truncated Wiener-Hopf operators (which may also be called matricial convolution operators on finite intervals). We show that we can use the theory developed around the modified MTTO to test the continuity of matricial truncated Wiener-Hopf operators. Finally, we show how the matricial truncated Wiener-Hopf operators are naturally encountered when finding the solution to multi input, multi output linear systems.

\section{The modified matrix-valued truncated Toeplitz operator}\label{sec:preliminaryresults}

The Riesz projections $P_{q+} : (L^q)^n \to (H^q)^n$ and $P_{q-} := I - P_{q+} : (L^q)^n \to ( \overline{H^q_0} )^n$ are bounded when $q \in ( 1, \infty )$. Furthermore for $n=1$, $P_{q+}$ can be expressed by 
$$
P_{q+} (f)(z) = \int_{\mathbb{T}} \frac{f (\zeta)}{1 - \overline{\zeta}z} dm( \zeta),
$$
which is independent of $q \in ( 1, \infty)$. Which means we can deduce the following;
\begin{lem}\label{02}
For $q \in ( 1, 2)$ and $f \in (L^2)^n$, we have $P_{q+} (f)  = P_{2+} (f) $ and  $P_{q-} (f)  = P_{2-} (f)$.
\end{lem}
\noindent Using the projection $ P_{\Theta,q} := P_{q+} \Theta P_{q-} \Theta^* $ we can decompose, as we have done in the case of $q = 2$, $(H^q)^n$ as
$$
(H^q)^n = K_{\Theta}^q \oplus \Theta (H^q)^n,
$$
where $K_{\Theta}^q = \Theta\overline{(H^q_0)^n} \cap (H^q)^n$ 
and $$
(L^q)^n = (\overline{H^q_0})^n \oplus K_{\Theta}^q \oplus \Theta (H^q)^n.
$$
As the orthogonal projection $ P_{\Theta} = P_{\Theta,2} : L^2 \to K_{\Theta}^2$ can also be written $P_{\Theta,2} = P_{2+} \Theta P_{2-} \Theta^*$ we can conclude the following;
\begin{lem}\label{0}
For $q \in (1 , 2) $ and $ f \in (L^2)^n$, we have 
$$
P_{\Theta,2}(f) = P_{\Theta,q}(f).
$$
\end{lem}
\noindent We can also deduce that for $Q_{\Theta,q}:= \Theta P_{q+} \Theta^* =  P_{q+} - P_{\Theta,q}: (L^q)^n \to \Theta (H^q)^n$ we have the following;
\begin{lem}\label{01}
For $q \in (1 , 2) $ and $ f \in (L^2)^n$, we have 
$$
Q_{\Theta,2}(f) = Q_{\Theta,q}(f).
$$
\end{lem}

\noindent We will use Lemmas \ref{02}, \ref{0} and \ref{01} freely throughout. When we are clearly working in the context of $(L^2)^n$, we will just write $P_{\Theta}$, respectively $Q_{\Theta}$, as opposed to $P_{\Theta,2}$, respectively $Q_{\Theta,2}$.

\vskip 0.5cm

\begin{defn}\label{tildemapdef}
Let $p \in ( 2, \infty ]$, let $G \in L^{( p, n \times n)}$ and let $\frac{1}{2} + \frac{1}{p} = \frac{1}{q} $. Then the bounded operator $\Tilde{ A_G^{\Theta}} : K_{\Theta}^2 \to K_{\Theta}^q$ is defined by $\Tilde{ A_G^{\Theta}}(f) = P_{\Theta,q} (Gf)$. We call the operator $\tilde{A_G^{\Theta}}$ the \textit{modified matrix-valued truncated Toeplitz operator}.
\end{defn}

\begin{rem}
Although $\tilde{A_G^{\Theta}}$ does have a specific $p$ dependence depending on which space $G$ lies in, we will omit this from our notation.
\end{rem}

\noindent The following proposition shows that when $A_G^{\Theta}: K_{\Theta}^2 \to K_{\Theta}^2$ is a MTTO, up to a change in codomain, $A_G^{\Theta}$ and $\Tilde{A_G^{\Theta}}$ are actually the same operator. In the next section we will exploit this link to study the kernel of $ A_G^{\Theta}$.

\begin{prop}\label{sameimg}
Let the assumptions of Definition \ref{tildemapdef} hold and let $A_G^{\Theta}: K_{\Theta}^2 \to K_{\Theta}^2$ be a MTTO. Then for each $f \in K_{\Theta}^2$ we have $\Tilde{ A_G^{\Theta}}(f) =A_G^{\Theta}(f) $.
\end{prop}

\begin{proof}
For a given $f \in K_{\Theta}^2$, let $f_n \in K_{\Theta}^2 \cap (H^{\infty})^n$ be such that $f_n \overset{(L^2)^n}{\to} f$. It is easily checked that $\Tilde{ A_G^{\Theta}}$ is bounded and so we have $P_{\Theta,q}(G f_n) \overset{(L^q)^n}{\to} P_{\Theta,q}(G f)$. By Lemma \ref{0} this means
\begin{equation}\label{star1}
    P_{\Theta,2}(G f_n) \overset{(L^q)^n}{\to} P_{\Theta,q}(G f) = \Tilde{ A_G^{\Theta}}(f).
\end{equation}
Because $P_{\Theta,2}(G f_n) \overset{(L^2)^n}{\to} P_{\Theta,2}(G f) = A_G^{\Theta}(f)$ and convergence in $(L^2)^n$ is stronger than $(L^q)^n$ we must have
\begin{equation}\label{star2}
P_{\Theta,2}(G f_n) \overset{(L^q)^n}{\to} P_{\Theta,2}(G f) = A_G^{\Theta}(f).
\end{equation}
Now comparing \eqref{star1} and \eqref{star2} uniqueness of limits implies that $\Tilde{ A_G^{\Theta}}(f) =A_G^{\Theta}(f) $.
\end{proof}

\begin{cor}\label{containedinL2}
Let the assumptions of Definition \ref{tildemapdef} hold and let $A_G^{\Theta}: K_{\Theta}^2 \to K_{\Theta}^2$ be a MTTO. Then $\text{Im}\Tilde{ A_G^{\Theta}} \subseteq K_{\Theta}^2 $.
\end{cor}

\noindent In fact we have the following;
\begin{prop}\label{boundedmttocondition}
Let the assumptions of Definition \ref{tildemapdef} hold. Then $\text{Im}\Tilde{ A_G^{\Theta}} \subseteq K_{\Theta}^2 $ if and only if $A_G^{\Theta}$ is a MTTO (i.e the map \eqref{MTTO} is bounded). 
\end{prop}

\begin{proof}
The above corollary shows that when $A_G^{\Theta}: K_{\Theta}^2 \to K_{\Theta}^2$ is a MTTO, then $\text{Im}\Tilde{ A_G^{\Theta}} \subseteq K_{\Theta}^2 $. To show the other implication, we first change the codomain of $\tilde{A_G^{\Theta}}$, to view the map $\tilde{A_G^{\Theta}}: K_{\Theta}^2 \to K_{\Theta}^2$, which is well defined by the assumption $\text{Im}\Tilde{ A_G^{\Theta}} \subseteq K_{\Theta}^2 $. We now use the closed graph theorem to show $\tilde{A_G^{\Theta}}: K_{\Theta}^2 \to K_{\Theta}^2$ is continuous. Let $(f_n)_{n \in \mathbb{N}} \in K_{\Theta}^2$ and let 
$$
( f_n , \tilde{A_G^{\Theta}} (f_n) ) \overset{K_{\Theta}^2 \times K_{\Theta}^2}{\to} (y_1, y_2),
$$
then clearly $f_n \overset{K_{\Theta}^2}{\to} y_1 $ and $\tilde{A_G^{\Theta}}(f_n) \overset{K_{\Theta}^2}{\to} y_2 $. We also know that $\tilde{A_G^{\Theta}}(f_n) \overset{K_{\Theta}^q}{\to} \tilde{A_G^{\Theta}}(y_1) $, and as $L^2$ convergence is stronger than $L^q$ convergence we can say that $\tilde{A_G^{\Theta}}(f_n) \overset{K_{\Theta}^q}{\to} y_2 $. Uniqueness of limits now shows $( f_n , \tilde{A_G^{\Theta}} (f_n) ) \overset{K_{\Theta}^2 \times K_{\Theta}^2}{\to} (y_1, \tilde{A_G^{\Theta}} (y_1 )) $, and hence the graph is closed. Now, again viewing $\tilde{A_G^{\Theta}}: K_{\Theta}^2 \to K_{\Theta}^2$, we have 
$$
\tilde{A_G^{\Theta}}(f) = A_G^{\Theta}(f)
$$
for all $ f \in K_{\Theta}^2 \cap (H^{\infty})^n$. Thus boundedness of $\tilde{A_G^{\Theta}}: K_{\Theta}^2 \to K_{\Theta}^2$ ensures boundedness of \eqref{MTTO}.
\end{proof}
In \cite{baranov2010symbols} the authors give an equivalent condition for a bounded truncated Toeplitz operator to have a bounded symbol. They then go on to describe the (scalar) inner functions, $\theta$, such that every bounded truncated Toeplitz operator on $K_{\theta}^2$ has a bounded symbol. If we consider MTTOs with symbols in $ L^{( p, n \times n)}$ where $p \in ( 2, \infty ]$, the above proposition allows one to describe the set of all symbols of  MTTOs (or when specialised to the scalar case, symbols of all bounded truncated Toeplitz operators). This is given by 
$$
\{ G : P_{\Theta, q} ( G K_{\Theta}^2) \subseteq (L^2)^n \} .
$$

In a similar fashion to how we have changed the codomain of the MTTO to obtain the modified MTTO, we can also change the codomain of the matricial Toeplitz operator. Let $p \in ( 2, \infty ]$ and let $G \in L^{( p, n \times n)}$. Define \begin{equation}\label{calG}
\mathcal{G} =\begin{pmatrix}
{\Theta}^* & 0 \\
G & \Theta \\
\end{pmatrix},
\end{equation}
where $0$ denotes the $n$-by-$n$ matrix with each entry being 0. Throughout all sections, given two Banach spaces $X_1, X_2$ we will equip the space 
$$
\begin{pmatrix}
X_1 \\
X_2
\end{pmatrix} = \left\{ \begin{pmatrix}
f_1 \\
f_2
\end{pmatrix} : f_1 \in X_1, f_2 \in X_2 \right\}
$$
with the norm given by 
$$
\left |\left |\begin{pmatrix}
f_1 \\
f_2
\end{pmatrix} \right |\right| = ||f_1||_{X_1} + || f_2 ||_{X_2}.
$$
With this convention we can define $T_{\mathcal{G}}: \begin{pmatrix}
(H^2)^n \\
(H^q)^n
\end{pmatrix}
\to \begin{pmatrix}
(H^2)^n \\
(H^q)^n
\end{pmatrix}$, where if $ f_1 \in (H^2)^n $ and $f_2 \in (H^q)^n$ (where $\frac{1}{2} + \frac{1}{p} = \frac{1}{q} $),
\begin{equation}\label{T_G}
    \begin{pmatrix}
f_1 \\
f_2
\end{pmatrix} \mapsto 
\begin{pmatrix}
P_{2+}({\Theta}^* f_1) \\
P_{q+}( G f_1 + \Theta f_2 )
\end{pmatrix}.
\end{equation}

\noindent An application of Hölder's inequality shows $T_{\mathcal{G}}$ is bounded. 
\begin{prop}\label{kercoordinates}
For the matrix $\mathcal{G}$ defined as \eqref{calG} we have $P_n (\ker T_{\mathcal{G}}) = \ker \Tilde{A_G^{\Theta}}$.
\end{prop}

\begin{proof}
Clearly, for $f_1 \in (H^2)^n$ and $f_2 \in (H^q)^n$, we have $(f_1,f_2) \in \ker T_{\mathcal{G}}$ if and only if
$f_1 \in \ker T_{\Theta^*} = K_{\Theta}$ and ${G} f_1
+ \Theta f_2  \in (\overline{H^q_0})^n.$
So $f_1 \in \ker A^\Theta_G$, and likewise given $f_1 \in \ker A^\Theta_G$ there exist $f_2 \in (H^q)^n$
with $(f_1,f_2) \in \ker T_{\mathcal{G}}$.
\end{proof}

Although the above results are interesting in their own right, our main motivation for introducing the modified MTTO is to study the properties of the MTTOs which do not posses a bounded symbol. 

The condition that we no longer require a bounded symbol to study $A_G^{\Theta}$ is a significant extension to previous studies. This is because there are MTTOs which do not have a bounded symbol but do have a symbol in $L^{(p, n \times n )}$, where $p \in (2, \infty)$. This can be shown in the case where $n=1$ by using Theorem 5.3 in \cite{baranov2010bounded}, which is the following;

\begin{thm}
Suppose $\theta$ is a (scalar) inner function which has an angular derivative (or ADC for short) at $ \zeta \in \mathbb{T}$. Let $p \in ( 2, \infty)$. Then the following are equivalent:
\begin{enumerate}
    \item the bounded truncated Toeplitz operator $k_{\zeta}^{\theta} \otimes k_{\zeta}^{\theta}$ has a symbol $\phi \in L^p$ ;
    \item $k_{\zeta}^{\theta} \in L^p$.
\end{enumerate}

\end{thm}
\noindent In the above Theorem $k_{\zeta}^{\theta} = \frac{1 - \overline{\theta(\zeta)}\theta (z)}{1 - \overline{\zeta} z} \in K_{\theta}$ is the reproducing kernel at $\zeta$. In particular, the above theorem shows that if $2 < p_1 < p_2 < \infty$ and $k_{\zeta}^{\theta} \in L^{p_1}$ but $k_{\zeta}^{\theta} \notin L^{p_2}$, then $k_{\zeta}^{\theta} \otimes k_{\zeta}^{\theta}$ does not have a bounded symbol but does have a symbol in $L^{p_1}$.

The precise conditions for $k_{\zeta}^{\theta}$ to lie in $L^p$ for $p \in (1, \infty)$ are given in \cite{1} and \cite{13}. In particular, for a Blaschke product with zeros $( a_k )$ we have $k_{\zeta}^{\theta} \in L^p $ if and only if 
\begin{equation}\label{2.5}
    \sum_{k} \frac{1 - |a_k|^2}{|\zeta - a_k|^p} < \infty.
\end{equation}

To obtain a bounded truncated Toeplitz operator which does not have a bounded symbol but does have a symbol in $L^{p_1}$, for some $p_1 \in (2, \infty)$, it is sufficient to have a point $\zeta \in \mathbb{T}$, and a Blaschke product which has an ADC at $\zeta$ such that \eqref{2.5} is true for some $p = p_1 \in (2, \infty)$ but not for some strictly larger value of $p$. An explicit example of this is a Blaschke product with zeros $(a_k)$ accumulating to the point 1 such that 
$$
\sum_{k} \frac{1 - |a_k|^2}{|1 - a_k|^{p_1}} < \infty \hskip 0.4cm \text{   for some $2 <p_1  < \infty $},
$$
but
$$
\sum_{k} \frac{1 - |a_k|^2}{|1 - a_k|^{p_2}} = \infty \hskip 0.4cm \text{   for some $p_1 < p_2 < \infty $} .
$$

Theorem 5.1(b) in \cite{sarason2007algebraic} states that if $\theta$ has an ADC at $\zeta \in \mathbb{T}$, then  $k_{\zeta}^{\theta} \otimes k_{\zeta}^{\theta}$ is a bounded truncated Toeplitz operator. Therefore by Theorem 5.1(b) in \cite{sarason2007algebraic} and the above theorem, we can construct an example of a bounded truncated Toeplitz operator which has a symbol in $L^2$, but does not have a symbol in $L^p$ for any $p \in ( 2 , \infty)$. Similar to our previous example, in order to do this it is sufficient to have a point $\zeta \in \mathbb{T}$ and a Blaschke product with an ADC at $\zeta$ such that \eqref{2.5} is true for $p = 2$ but not for any $p \in ( 2 , \infty )$. A numerical example of such a point $\zeta \in \mathbb{T}$ and  Blaschke product is the Blaschke product with zeros (accumulating to 1) given by $a_k = (1 - \epsilon_k)e^{i \delta_k}$ where $\epsilon_k = \frac{1}{k^2} $ and $\delta_k = \frac{log(k)}{k^{1/2}}$. This observation shows that not every bounded truncated Toeplitz operator has a symbol in $L^p$ for some $p \in (2, \infty)$.

\section{The kernel}\label{sec:3.1}

A closed subspace $M \subseteq (H^2)^n $ is said to be nearly $S^*$-invariant with defect $m$ if and only if there exists a $m$-dimensional subspace $D$ (which may be taken to be orthogonal to $M$) such that if $f \in M$ and $f(0) $ is the zero vector then $S^* f \in M \oplus D$. We call $D$ the defect space. If $M$ is nearly $S^*$-invariant with defect 0 then it is said to be nearly $S^*$-invariant. Similarly, we say a closed subspace $N \subseteq \begin{pmatrix}
(H^2)^n \\
(H^q)^n
\end{pmatrix}$ is nearly $S^*$-invariant if and only if all functions $f \in N$ with the property $f(0)$ is the zero vector satisfy $S^*(f) = \frac{f}{z} \in N$.

In this section we decompose the kernel of a MTTO into an isometric image of an $S^*$-invariant subspace.

\vskip 0.5cm

\noindent Define $W :=  \ker T_{\mathcal{G}} (0) = \{ F(0) : F \in \ker T_{\mathcal{G}} \} \subseteq \mathbb{C}^{2n}$. Let $\dim W = r$, and pick $W_1, \hdots ,W_r \in \ker T_{\mathcal{G}}$ such that $W_1(0), \hdots ,W_r(0)$ are a basis for W.

\begin{prop}\label{generalninv}
The space $P_{n} ( \ker T_{\mathcal{G}})$ is nearly $S^*$-invariant with a defect space
\begin{equation}\label{defect}
    \left( \frac{\text{span} \{ P_n(W_1), \hdots  P_n(W_r) \} }{z} \cap  (H^2)^n \right).
\end{equation}
\end{prop}
\begin{rem}
This may be viewed as a generalisation of Corollary 3.2 in \cite{oloughlin2020nearly}, but the delicate issue here is that we are no longer working with a Hilbert space and so we can not use orthogonality.
\end{rem}

\begin{proof}
Let $f_1 \in P_n ( \ker T_{\mathcal{G}} )$ with $f_1(0)$ equal to the zero vector. Pick $f_2 \in (H^q)^n$ such that $\begin{pmatrix}
f_1 \\
f_2
\end{pmatrix} \in \ker T_{\mathcal{G}}$ and pick constants $ \lambda_1 \hdots \lambda_r $ such that $\begin{pmatrix}
f_1 \\
f_2
\end{pmatrix} - \lambda_1 W_1 - \hdots \lambda_r W_r$ evaluated at $0$ is the zero vector, then
$$
\begin{pmatrix}
f_1 \\
f_2
\end{pmatrix} - \lambda_1 W_1 - \hdots \lambda_r W_r \in \ker T_{\mathcal{G}} \cap z \begin{pmatrix}
(H^2)^n \\
(H^q)^n
\end{pmatrix}. 
$$
Near invariance of $\ker T_{\mathcal{G}}$ now ensures 
$$
\frac{f_1}{z} - \frac{\lambda_1 P_{n}(W_1) - \hdots \lambda_r P_{n}(W_r)}{z} \in P_{n}( \ker T_{\mathcal{G}}),
$$
and therefore 
$$
\frac{f_1}{z} \in P_{n}( \ker T_{\mathcal{G}}) + \left( \frac{\text{span} \{ P_n(W_1), \hdots  P_n(W_r) \} }{z} \cap  (H^2)^n \right).
$$
\end{proof}

Previous results on the kernel of the truncated Toeplitz operator (see \cite{oloughlin2020nearly}, \cite{scalartype} and \cite{MR3634513}) have been under the assumption that the symbol for the operator is bounded. Now using the operator $\Tilde{A_G^{\Theta}}$ as an intermediate tool, this allows us to obtain a Hitt-style characterisation for the kernel of a MTTO and, unlike previous results, we do not require that the symbol of the MTTO is bounded for this characterisation to hold.

\begin{thm}\label{MTTOnear}
Let $p \in (2, \infty ]$, and let $G \in L^{(p, n \times n )}$ be such that $A_G^{\Theta}$ is a MTTO. Then $\ker A_G^{\Theta}$ is nearly $S^*$-invariant with defect $m$, where $m \leqslant n$.
\end{thm}

\begin{proof}
From Proposition \ref{sameimg} it is clear that $\ker A_G^{\Theta} =\ker \Tilde{A_G^{\Theta}}$, and Proposition \ref{kercoordinates} shows that $\ker \Tilde{A_G^{\Theta}} = P_{n} (\ker T_{\mathcal{G}})$, so from Proposition \ref{generalninv} we can deduce that $\ker A_G^{\Theta}$ is a nearly invariant subspace with a defect space given by \eqref{defect}. If $r \leqslant n$ it is clear that the dimension of \eqref{defect} is less than or equal to $n$, so it remains to prove that if $r = n + i$ for $i > 0$ then the dimension of \eqref{defect} is at most $n$. Suppose $r = n + i$ for $i > 0$. We form a matrix
$$
[W_1(0),\hdots, W_{n+i}(0)],
$$
then for $\begin{pmatrix}
s_1 \\
\vdots \\
s_{n+i}
\end{pmatrix} \in \mathbb{C}^{n+i}$ we have that $s_1 P_n (W_1) + \hdots s_{n+i} P_{n} (W_{n+i}) \in z (H^2)^n$ if and only if
$$
P_n \left( [W_1(0),\hdots, W_{n+i}(0)] \begin{pmatrix}
s_1 \\
\vdots \\
s_{n+i}
\end{pmatrix} \right)
$$
is the zero vector. Hence the dimension of \eqref{defect} is given by the dimension of 
$$
S = \left\{ \begin{pmatrix}
s_1 \\
\vdots \\
s_{n+i}
\end{pmatrix} \in \mathbb{C}^{n+i} : P_n \left( [W_1(0),\hdots, W_{n+i}(0)] \begin{pmatrix}
s_1 \\
\vdots \\
s_{n+i}
\end{pmatrix} \right) = \begin{pmatrix}
0 \\
\vdots \\
0\\
\end{pmatrix} \right\}.
$$

As $W_1(0), ... W_{n+i}(0) \in \mathbb{C}^{2n}$ are linearly independent, we may pick vectors $$X_1, \hdots X_{n-i}  \in \mathbb{C}^{2n}$$ such that the vectors $W_1(0), \hdots $ $W_{n+i}(0), X_1, \hdots, X_{n-i}$ are linearly independent. We then define $S^{'}$ as
$$
\left\{ \begin{pmatrix}
s_1 \\
\vdots \\
s_{n+i} \\
0 \\
\vdots \\
0
\end{pmatrix} \in \mathbb{C}^{2n} : P_n \left( [W_1(0), \hdots ,W_{n+i}(0), X_1, \hdots, X_{n-i}] \begin{pmatrix}
s_1 \\
\vdots \\
s_{n+i} \\
0 \\
\vdots \\
0
\end{pmatrix} \right) = \begin{pmatrix}
0 \\
\vdots \\
\vdots \\
\vdots \\
0
\end{pmatrix} \right\}.
$$
It is clear $\dim S = \dim S^{'}$, and moreover $S^{'}$ is contained in 
$$
\left\{ [W_1(0), \hdots ,W_{n+i}(0), X_1, \hdots, X_{n-i}]^{-1} V : V \in \mathbb{C}^{2n} \text{and } P_n (V) = 0 \right\},
$$
which has dimension $n$. Thus we can conclude that the dimension of \eqref{defect} is equal to $\dim S = \dim S^{'} \leqslant n$.
\end{proof}

Theorem 3.4 in \cite{oloughlin2020nearly} (which was also independently proved in \cite{chattopadhyay2020invariant}) gives a decomposition for vector-valued nearly $S^*$-invariant subspaces with a defect. So combining the above theorem and Theorem 3.4 in \cite{oloughlin2020nearly} we obtain the following decomposition for the kernels of MTTOs in terms of $S^{*}$-invariant subspaces.

\begin{thm}\label{decompfornear}
Let $p \in (2, \infty ]$, and let $G \in L^{(p, n \times n )}$ be such that $A_G^{\Theta}$ is a MTTO. Let $\{ e_1 , \hdots e_m \}$ be an orthonormal basis for the $m$-dimensional defect space (where $m \leqslant n$) for $\ker A_G^{\Theta}$ given by \eqref{defect} and set $r = \dim \ker A_G^{\Theta} \ominus ( \ker A_G^{\Theta} \cap z (H^2)^n)$. Then \begin{enumerate}
     \item in the case where there are functions in $\ker A_G^{\Theta}$ that do not vanish at $0$,
     $$
     \ker A_G^{\Theta} = \{ F : F(z) = F_0 (z) k_0(z) + z \sum_{j=1}^{m} k_j(z) e_j (z) : ( k_0, \hdots ,k_m) \in K \} ,
     $$
     where $F_0$ is the matrix with each column being an element of an orthonormal basis for $ \ker A_G^{\Theta} \ominus ( \ker A_G^{\Theta} \cap z (H^2)^n)$, $k_0 \in (H^2)^r$, $k_1, \hdots k_m \in H^2$, and $K \subseteq (H^2)^{(r+m)}$ is a closed $S^*$-invariant subspace. Furthermore $ ||F||^2 = \sum_{j=0}^m ||k_j||^2$.
     \item In the case where all functions in $\ker A_G^{\Theta}$ vanish at $0$,
     $$
     \ker A_G^{\Theta} = \{ F : F(z) = z \sum_{j=1}^{m} k_j(z) e_j(z) : (k_1, \hdots, k_m) \in K \},
     $$
     with the same notation as in 1, except that $K$ is now a closed $S^*$-invariant subspace of $(H^2)^m$, and $||F||^2 = \sum_{j=1}^{m} ||k_j||^2$.
 \end{enumerate}
\end{thm}

\begin{rem}
We remark that the above theorem is a generalisation of the results of Section 3 of \cite{oloughlin2020nearly} in two ways. We are now considering the MTTO instead of the scalar truncated Toeplitz operator. We are also now allowing for the MTTO to have a unbounded symbol whereas \cite{oloughlin2020nearly} only considers bounded symbols.
\end{rem}

We now give an example to show that under the conditions of Theorem \ref{MTTOnear}, $n$ is the smallest dimension of defect space for $\ker A_{G}^{\Theta}$, i.e. it is not true that for all inner functions $\Theta$ and symbols $G \in L^{(p, n \times n)}$, that $\ker A_{G}^{\Theta}$ has a $j$-dimensional defect where $j < n$.

\begin{ex}
Let $\Theta = 
\begin{pmatrix}
z^2 & 0 \\
0 & z^2 \\
\end{pmatrix}$, and $G = \begin{pmatrix}
z & 0 \\
0 & z \\
\end{pmatrix}$ then 
$$
\ker A_G^{\Theta} = \left\{ \begin{pmatrix}
\lambda z  \\
\mu z  \\
\end{pmatrix} : \lambda, \mu \in \mathbb{C} \right\},
$$
which is clearly nearly $S^*$-invariant with defect 2.
\end{ex}

\section{Equivalence after extension}\label{sec:4}
For Banach spaces $X, \tilde{X} , Y , \tilde{Y}$ the operators $T: X \to \tilde{X}$ and $S: Y \to \tilde{Y}$ are said to be (algebraically and topologically) equivalent if and only if $T = ESF$, where $E$ and $F$ are invertible operators. More generally $T$ and $S$ are equivalent after extension (abbreviated to EAE) if and only if there exist (possibly trivial) Banach spaces $X_0, Y_0$, called extension spaces and invertible linear operators $E : \tilde{Y} \oplus Y_0 \to \tilde{X} \oplus X_0 $ and $F: X \oplus X_0 \to Y \oplus Y_0$, such that 
$$
\begin{pmatrix}
T & 0 \\
0 & I_{X_0} \\
\end{pmatrix} =  E \begin{pmatrix}
S & 0 \\
0 & I_{Y_0} \\
\end{pmatrix} F .
$$

\noindent In this case we write that $T \overset{*}{\backsim} S$.

\vskip 0.5cm
The relation $\overset{*}{\backsim}$ is an equivalence relation. Operators that are equivalent after extension have many features in common. In particular, using the notation $X \simeq Y$ to say that two Banach spaces $X$ and $Y$ are isomorphic, i.e., that there exists an invertible operator from $X$ onto $Y$ , and the notation $\text{Im} A$
to denote the range of an operator $A$, we have the following.

\begin{thm}[\cite{bartmatricial}]\label{EAEbackgrond}
Let $T: X \to \tilde{X}, S: Y \to \tilde{Y}$ be operators and assume that $T \overset{*}{\backsim} S$. Then
\begin{enumerate}
    \item $\ker T \simeq \ker S$;
    \item $\text{Im }T$ is closed if and only if $\text{Im }S$ is closed and, in that case, $\tilde{X}/\text{Im }T \simeq \tilde{Y}/\text{Im }S$;
    \item if one of the operators $T, S$ is generalised (left, right) invertible, then the other is generalised (left, right) invertible too;
    \item $T$ is Fredholm if and only if $S$ is Fredholm and in that case $\dim \ker T = \dim \ker S$ and $\text{codim } \text{Im }T = \text{codim } \text{Im }S$.
\end{enumerate}
\end{thm}

The above theorem highlights that when one wants to consider invertibility, Fredholmness and spectral properties, EAE extension results are very useful. Section 6 of \cite{MR3634513} shows that a truncated Toeplitz operator with a bounded symbol is EAE to a matricial Toeplitz operator, and then consequently the spectral properties of the truncated Toeplitz operator were studied in \cite{MR3398735}. For $\theta$ a scalar inner function and $g \in L^{\infty}$, the dual truncated Toeplitz operator $D_g^{\theta}: (K_{\theta})^{\perp} \to (K_{\theta})^{\perp}$ is defined by $f \mapsto (Q_{\theta} + P_{-}) (gf)$, where $P_- = I - P_+$. Section 5 of \cite{camara2019invertibility} shows the dual truncated Toeplitz operator is EAE to a paired operator on $(L^2)^2$.

\vskip 1cm
\noindent Throughout this section, unless otherwise stated, we assume that $G \in L^{( p, n \times n)}$ where $p \in (2 , \infty]$. We let $ q \in (1, 2]$ be such that $\frac{1}{2} + \frac{1}{p} = \frac{1}{q}$. In this context, we write $T_G :  (H^2)^n \to (H^q)^n$ to mean the map $f \mapsto P_{q+} (Gf)$.

\vskip 1cm

\noindent In the first part of this section we initially adapt the results in Section 6 of \cite{MR3634513} to show that $T_{\mathcal{G}}$ is EAE to $\Tilde{ A_G^{\Theta}}$. We then build on this result to construct a Toeplitz operator which is EAE to $A_G^{\Theta}$. Unlike the works of \cite{MR3634513} we consider MTTOs which only have unbounded symbols, and in order to overcome the problem of $G$ not being bounded (and then necessarily the domain and codomain of $\Tilde{A_G^{\Theta}}$ being different spaces) one must define a new normed space which mixes $H^p$ and $H^q$ spaces.

Consider the operator $$P_{\Theta,q}G P_{\Theta,2} + Q_{\Theta,2}: (H^2)^n \to K_{\Theta}^q + \Theta (H^2)^n, $$ where here the norm of $k + \Theta f \in K_{\Theta}^q + \Theta (H^2)^n$ is given by $||k||_{(L^q)^n} + ||\Theta f ||_{(L^2)^n}$. We first show that 
\begin{equation}\label{a}
    \Tilde{A_G^{\Theta}} \overset{*}{\backsim} P_{\Theta,q}G P_{\Theta,2} + Q_{\Theta,2}.
\end{equation}

\noindent We have
$$
\begin{pmatrix}
\Tilde{A_G^{\Theta}} & 0 \\
0 & I_{\Theta (H^2)^n} \\
\end{pmatrix} = E_1 \begin{pmatrix}
P_{\Theta,q}G P_{\Theta,2} + Q_{\Theta,2} & 0 \\
0 & I_{ 0 } \\
\end{pmatrix} F_1 ,
$$
where
$$
F_1 : K_{\Theta}^2 \oplus \Theta (H^2)^n \to (H^2)^{n} \oplus \{ 0 \}
$$
is such that
$$
\begin{pmatrix}
k \\
\Theta f \\
\end{pmatrix} \mapsto \begin{pmatrix}
k + \Theta f \\
0 \\
\end{pmatrix} ,
$$
and
$$
E_1 : K_{\Theta}^q + \Theta (H^2)^n \oplus \{ 0 \} \to K_{\Theta}^q \oplus \Theta (H^2)^n 
$$
is such that 
$$
\begin{pmatrix}
k + \Theta f \\
0 \\
\end{pmatrix} \mapsto \begin{pmatrix}
k \\
\Theta f \\
\end{pmatrix}.
$$

\noindent On the other hand it is clear that 
\begin{equation}\label{b}
   P_{\Theta,q}G P_{\Theta,2} + Q_{\Theta,2} \overset{*}{\backsim} \begin{pmatrix}
P_{\Theta,q}G P_{\Theta,2} + Q_{\Theta,2} & 0 \\
0 & P_{q+} \\
\end{pmatrix}.
\end{equation}

\noindent If we denote $I$ to be the identity operator on $ K_{\Theta}^q + \Theta (H^2)^n$, we also have
$$
P_{\Theta,q}G P_{\Theta,2} + Q_{\Theta,2} = ( I - P_{\Theta,q} T_G Q_{\Theta,q})(P_{\Theta,q}T_G + Q_{\Theta,2}).
$$
Furthermore adapting Lemma 6.3 in \cite{MR3634513} we can deduce:

\begin{lem}\label{first}
The operator $ I - P_{\Theta,q} T_G Q_{\Theta,q}:  K_{\Theta}^q + \Theta (H^2)^n \to  K_{\Theta}^q + \Theta (H^2)^n $ is invertible with inverse $  I + P_{\Theta,q} T_G Q_{\Theta,q}$.
\end{lem}

We now mimic the factorisations given in Section 6 of \cite{MR3634513}, however as we are working with a mixed $H^p$-$H^q$ space we must also manage our factorisation in such a way that the domain and codomain in our consecutive factors match up. As the results of \cite{MR3634513} purely deal with operators on $H^2$, this did not have to be considered.

In the following argument for ease of notation we write the domain and co-domain above the operator. For example, if the operator $A: X \to Y$, we will label this as $\overbrace{A}^{X \to Y}$. In the case when $A: X \to X$ we will denote this by $\overbrace{A}^{X}$. With this notation we will omit the specific $q$ or $2$ notation from the projections in the following matrices.

Thus with 
\begin{equation}
T = \begin{pmatrix}
\overbrace{I - P_{\Theta}T_G Q_{\Theta} }^{K_{\Theta}^{q} + \Theta (H^2)^n} & \overbrace{0}^{(H^q)^n \to \{0\}} \\
\overbrace{0}^{K_{\Theta}^q + \Theta (H^2)^n \to \{0\}} & \overbrace{P_+}^{ (H^q)^n} \\
\end{pmatrix},
\end{equation}
we can write
$$
\begin{pmatrix}
\overbrace{P_{\Theta}G P_{\Theta} + Q_{\Theta}}^{(H^2)^n \to K_{\Theta}^q + \Theta (H^2)^n} & \overbrace{0}^{(H^q)^n \to K_{\Theta}^q + \Theta (H^2)^n} \\
\overbrace{0}^{(H^2)^n \to (H^q)^n} & \overbrace{P_+}^{(H^q)^n} \\
\end{pmatrix} =
$$
$$T  \begin{pmatrix}
\overbrace{P_{\Theta} T_G + Q_{\Theta}}^{(H^2)^n \to K_{\Theta}^q + \Theta (H^2)^n} & \overbrace{0}^{(H^q)^n \to K_{\Theta}^q + \Theta (H^2)^n} \\
\overbrace{0}^{(H^2)^n \to (H^q)^n} & \overbrace{P_+}^{(H^q)^n} \\
\end{pmatrix} = 
$$
$$
 T \begin{pmatrix}
\overbrace{T_{\Theta}}^{(H^2)^n \to \Theta (H^2)^n} & \overbrace{P_{\Theta}}^{(H^q)^n \to K_{\Theta}^q} \\
\overbrace{-P_+}^{(H^2)^n} & \overbrace{T_{{\Theta}^*}}^{(H^q)^n} \\
\end{pmatrix} \begin{pmatrix}
\overbrace{T_{{\Theta}^*}}^{(H^2)^n} & \overbrace{0}^{(H^q)^n \to (H^2)^n} \\
\overbrace{{T_G - Q_{\Theta}T_G + Q_{\Theta}P_+}}^{(H^2)^n \to (H^q)^n} & \overbrace{T_{\Theta}}^{(H^q)^n} \\
\end{pmatrix},
$$
where the last line follows by using the identity $P_+ - Q_{\Theta} = P_{\Theta}$ and $T_{{\Theta}^*} P_{\Theta} = 0 $. This can be factorised further to equal
\begin{equation}\label{c}
      T \begin{pmatrix}
\overbrace{T_{\Theta}}^{(H^2)^n \to \Theta (H^2)^n} & \overbrace{P_{\Theta}}^{(H^q)^n \to K_{\Theta}^q} \\
\overbrace{-P_+}^{(H^2)^n} & \overbrace{T_{{\Theta}^*}}^{(H^q)^n} \\
\end{pmatrix} T_{\mathcal{G}} \begin{pmatrix}
\overbrace{P_+}^{(H^2)^n} & \overbrace{0}^{(H^q)^n \to (H^2)^n} \\
\overbrace{ -T_{{\Theta}^*}T_G + T_{{\Theta}^*}P_+ }^{(H^2)^n \to (H^q)^n} & \overbrace{P_+}^{(H^q)^n} \\
\end{pmatrix},
\end{equation}

\noindent where $T_{\mathcal{G}}$ is defined as in \eqref{T_G}. In the above, we label the second factor as $T_1$ and the final factor as $T_2$.

\begin{enumerate}
    \item The first factor, $T$, is invertible with inverse given by 
    $$
    \begin{pmatrix}
\overbrace{I + P_{\Theta}T_G Q_{\Theta} }^{K_{\Theta}^{q} + \Theta (H^2)^n} & \overbrace{0}^{(H^q)^n \to \{0\}} \\
\overbrace{0}^{K_{\Theta}^q + \Theta (H^2)^n \to \{0\}} & \overbrace{P_+}^{ (H^q)^n} \\
\end{pmatrix}.
    $$
    This is verified by Lemma \ref{first}.
    \item Adapting Lemma 6.4 from \cite{MR3634513} one can show the second factor, $T_1$, is invertible as a map $\begin{pmatrix}
    (H^2)^n \\
    (H^q)^n
    \end{pmatrix} \to \begin{pmatrix}
    K_{\Theta}^q + \Theta (H^2)^n \\
    (H^q)^n
    \end{pmatrix}$. 
    \item Adapting Lemma 6.5 from \cite{MR3634513} one can show the last factor, $T_2$, is invertible in $ \begin{pmatrix}
    (H^2)^n \\
    (H^q)^n
    \end{pmatrix} $.
\end{enumerate}


\noindent We can now conclude the following;
\begin{thm}\label{maineae}
$T_{\mathcal{G}}$ is equivalent after extension to $\Tilde{A_{G}^{\Theta}}$. 
\end{thm} 
\begin{proof}
Using \eqref{a}, \eqref{b} and the fact that $\overset{*}{\backsim}$ is an equivalence relation, we see that 
$$
\Tilde{A_{G}^{\Theta}} \overset{*}{\backsim} \begin{pmatrix}
P_{\Theta,q}G P_{\Theta,2} + Q_{\Theta,2} & 0 \\
0 & P_{q+} \\
\end{pmatrix}.
$$
Now \eqref{c} and the reasoning immediately following \eqref{c} shows 
$$
\begin{pmatrix}
P_{\Theta,q}G P_{\Theta,2} + Q_{\Theta,2} & 0 \\
0 & P_{q+} \\
\end{pmatrix} \overset{*}{\backsim} T_{\mathcal{G}}
$$
and so transitivity of $\overset{*}{\backsim}$ gives us 
$$
\Tilde{A_{G}^{\Theta}} \overset{*}{\backsim} T_{\mathcal{G}}.
$$

\end{proof}

\begin{rem}
In the case when $n=1$ and $p = \infty$, Theorem \ref{maineae} specialises to become (the symmetric case of) Theorem 6.6 in \cite{MR3634513}.
\end{rem}

\noindent When $G$ is bounded we have $\tilde{A_G^{\Theta}} = A_G^{\Theta}$, so we may specialise Theorem \ref{maineae} to find an operator which is EAE to $A_G^{\Theta}$ when $G$ is bounded.

\begin{thm}
Let $G \in L^{( \infty, n \times n)}$. Then $T_{\mathcal{G}}: (H^2)^{2n} \to (H^2)^{2n}$ is equivalent after extension to ${A_{G}^{\Theta}}$. 
\end{thm}

As operators which are EAE have isomorphic kernels and cokernels, Theorem \ref{maineae} and Proposition \ref{sameimg} suggest that restricting the codomain of $T_{\mathcal{G}}$ may provide an operator which is EAE to $A_G^{\Theta}$, where $G \in L^{( p, n \times n)}$, for $p \in (2, \infty)$. We now pursue this idea.

\vskip 0.4cm

\noindent Throughout the remainder of this section we now continue to assume that $G \in L^{( p, n \times n)}$ where $p \in (2 , \infty]$, but we now we also make the extra assumption that ${A_{G}^{\Theta}}$ is a MTTO (and hence bounded).

\vskip 0.4cm

\noindent The image of $T_{\mathcal{G}}$ is computed to be
$$
\begin{pmatrix}
0 \\
\Theta (H^q)^n \\
\end{pmatrix} +
\begin{pmatrix}
0 \\
P_{q+} ( G K_{\Theta}^2) \\
\end{pmatrix} +
\left\{    \begin{pmatrix}
f \\
P_{q+} (G\Theta f) \\
\end{pmatrix} : f \in (H^2)^n \right\},
$$
where for $A \subseteq (L^q)^n$, $\begin{pmatrix}
0 \\
A \\
\end{pmatrix} $ is the set of all vectors of length $2n$ with the last $n$ coordinates taking a value $a \in A$. We now define the Banach space
\begin{equation}\label{co-d}
    \text{Co-d} := \begin{pmatrix}
0 \\
\Theta (H^q)^n \\
\end{pmatrix} +
\begin{pmatrix}
0 \\
K_{\Theta}^2 \\
\end{pmatrix} +
\left\{    \begin{pmatrix}
f \\
P_{q+} (G\Theta f) \\
\end{pmatrix} : f \in (H^2)^n \right\},
\end{equation}
where for $p_1 \in (H^q)^n, p_2 \in K_{\Theta}^2, p_3 \in (H^2)^n$ we have the well defined norm
$$
\left| \left| \begin{pmatrix}
0 \\
\Theta p_1 \\
\end{pmatrix} +
\begin{pmatrix}
0 \\
p_2 \\
\end{pmatrix} + 
\begin{pmatrix}
p_3 \\
P_{q+} (G\Theta p_3) \\
\end{pmatrix} \right| \right|_{\text{Co-d}} := || \Theta p_1 ||_{(H^q)^n} + ||p_2||_{K_{\Theta}^2} + ||p_3||_{(H^2)^n} .
$$
We note that completeness of each of the spaces $(H^q)^n , K_{\Theta}^2$ and $(H^2)^n$ ensures completeness of $\text{Co-d}$. Corollary \ref{containedinL2} ensures that $P_{q+} ( G K_{\Theta}^2) \subseteq K_{\Theta}^2 + \Theta (H^q)^n$ so this gives us a well defined bounded map
$$
T_{\mathcal{G}}^r :\begin{pmatrix}
(H^2)^n \\
(H^q)^n
\end{pmatrix} \to \text{Co-d},
$$ where for $ f_1 \in (H^2)^n $ and $f_2 \in (H^q)^n$
$$
\begin{pmatrix}
f_1 \\
f_2
\end{pmatrix} \mapsto 
\begin{pmatrix}
P_{2+}({\Theta}^* f_1) \\
P_{q+}( G f_1 + \Theta f_2 )
\end{pmatrix} = T_{\mathcal{G}}\begin{pmatrix}
f_1 \\
f_2
\end{pmatrix}.
$$
\begin{rem}
In the case when $p = \infty$ and so $q = 2$, as sets we have $\text{Co-d} = \begin{pmatrix}
(H^2)^n \\
(H^2)^n
\end{pmatrix}$ and furthermore the $\text{Co-d}$ norm is equivalent to the $\begin{pmatrix}
(H^2)^n \\
(H^2)^n
\end{pmatrix}$ norm.
\end{rem}

Similar to the proof of Theorem \ref{maineae}, we can show that 
$$
{A_{G}^{\Theta}} \overset{*}{\backsim} \begin{pmatrix}
\overbrace{P_{\Theta,q}G P_{\Theta,2} + Q_{\Theta,2}}^{(H^2)^n \to K_{\Theta}^2 + \Theta (H^2)^n} & \overbrace{0}^{(H^q)^n \to K_{\Theta}^q + \Theta (H^2)^n} \\
\overbrace{0}^{(H^2)^n \to (H^q)^n} & \overbrace{P_{q+}}^{(H^q)^n} \\
\end{pmatrix},
$$
where we know by Corollary \ref{containedinL2} that $P_{\Theta,q} (G K_{\Theta}^2) \subseteq K_{\Theta}^2$. It is also clear that for $\begin{pmatrix}
f_1 \\
f_2
\end{pmatrix}  \in 
\begin{pmatrix}
(H^2)^n \\
(H^q)^n
\end{pmatrix}$ using \eqref{c} we still have 
$$
T T_1 T_{\mathcal{G}}^r T_2 \begin{pmatrix}
f_1 \\
f_2
\end{pmatrix}  = \begin{pmatrix}
\overbrace{P_{\Theta,q}G P_{\Theta,2} + Q_{\Theta,2}}^{(H^2)^n \to K_{\Theta}^2 + \Theta (H^2)^n} & \overbrace{0}^{(H^q)^n \to K_{\Theta}^q + \Theta (H^2)^n} \\
\overbrace{0}^{(H^2)^n \to (H^q)^n} & \overbrace{P_{q+}}^{(H^q)^n} \\
\end{pmatrix} \begin{pmatrix}
f_1 \\
f_2
\end{pmatrix} .
$$
One can also check that the operator $T T_1 : \text{Co-d} \to \begin{pmatrix}
    (H^2)^n \\
    (H^q)^n
    \end{pmatrix}$ is well defined, bounded and invertible. We know from an adaptation of Lemma 6.5 in \cite{MR3634513} that $T_2: \begin{pmatrix}
    (H^2)^n \\
    (H^q)^n
    \end{pmatrix} \to \begin{pmatrix}
    (H^2)^n \\
    (H^q)^n
    \end{pmatrix}$ is invertible. So we can conclude;
\begin{thm}
${A_{G}^{\Theta}} \overset{*}{\backsim} T_{\mathcal{G}}^r$.
\end{thm}

\section{Application to integral equations}\label{integralequation}
In this section we concretely relate the theory of MTTOs to integral equations. We show the matricial truncated Wiener-Hopf operator is unitarily equivalent to a MTTO and give a neat application of the theory we have developed around the modified MTTO which allows us to test continuity of the matricial truncated Wiener-Hopf operator by considering only the image of a modified MTTO.

We consider $H^2 ( \mathbb{C}_+ )$, the Hardy space of the upper half plane. Section 3 of \cite{MR3634513} shows there is an isometric isomorphism between the Hardy space on the disc and the Hardy space on the half plane, which we denote $V: H^2 \to H^2(\mathbb{C}_+ )$. It further shows that via the unitary map $V$, truncated Toeplitz operators in the two different settings are unitarily equivalent. The functions $\theta_1 ( w) = e^{iaw}$ and $\theta_2 ( w) = e^{ibw}$ are inner functions in $H^2 ( \mathbb{C}_+ )$ and it is well known that the inverse Fourier transform of $L^2 (0, a)$, $L^2 (0, b)$ is $K_{\theta_1}^2 $, $K_{\theta_2}^2$ respectively. We denote $\mathcal{F}: H^2(\mathbb{C}_+ ) \to L^2(0, \infty )$ to be the Fourier transform. The Fourier transform of a matrix is understood by taking the Fourier transform of each entry of the matrix.

MTTOs on the model space $K_{\Theta}^2$ where $\Theta = \begin{pmatrix}
\theta_1 & 0 \\
0 & \theta_2 \\
\end{pmatrix}$, are closely connected with matricial truncated Wiener-Hopf operators. For $c \in \mathbb{R}_+$ and $f = \begin{pmatrix}
f_1 \\
f_2
\end{pmatrix} $ where $f_1, f_2 \in L^1( \mathbb{R} )$, we denote $\text{int}_c^1 (f) = \int_0^c f_1(t) dt$ and $\text{int}_c^2 (f) = \int_0^c f_2(t) dt$. Let $G = \begin{pmatrix}
g_{11} & g_{12} \\
g_{21} & g_{22}
\end{pmatrix} \in L^1 ( \mathbb{R} )^{2 \times 2}$, let $k = \begin{pmatrix}
k_1 \\
k_2
\end{pmatrix}$ where $k_1 \in L^2 (0, a)$ and $k_2 \in L^2 (0, b)$ for $a,b \in \mathbb{R}_+$. The matricial truncated Wiener-Hopf operator is an operator on $\begin{pmatrix}
L^2(0,a) \\
L^2(0,b)
\end{pmatrix}$ (equipped with the product norm) densely defined by
\begin{equation}\label{MTWH}
    \left( W_G 
\begin{pmatrix}
k_1 \\
k_2
\end{pmatrix} \right) (x) = \begin{pmatrix}
\text{int}_a^1 (G(x-t) k(t)) \\
\text{int}_b^2 (G(x-t) k(t))
\end{pmatrix} ,
\end{equation}
for $k_1 \in L^2 (0, a) \cap L^{\infty} (0, a)$, $k_2 \in L^2 (0, b) \cap L^{\infty} (0, b)$. We mostly only consider the case when $a=b$. In this case
$$
\left( W_G 
\begin{pmatrix}
k_1 \\
k_2
\end{pmatrix} \right) (x) = \int_0^a G(x-t) k(t) dt .
$$

If $W$ extends to a bounded operator and if $G = \hat{H}$, where $H \mathcal{F}^{-1}\begin{pmatrix}
L^{\infty}(0,a) \\
L^{\infty}(0,b)
\end{pmatrix} \subseteq (L^2(\mathbb{R}))^2$ (this condition on $H$ is necessary to densely define a MTTO with symbol $H$ on the domain $\mathcal{F}^{-1}\begin{pmatrix}
L^{\infty}(0,a) \\
L^{\infty}(0,b)
\end{pmatrix})$  we have 
$$
W_G (k) = \mathcal{F} P_{\Theta} (H \check{k} ),
$$
for $k_1 \in L^2 (0, a) \cap L^{\infty} (0, a)$, $k_2 \in L^2 (0, b) \cap L^{\infty} (0, b)$. Thus, in this case $W_G$ is unitarily equivalent to the MTTO on the half plane $A_{H}^{\Theta}$, where we adopt the convention that $A_{H}^{\Theta}$ is initially densely defined on $\mathcal{F}^{-1}\begin{pmatrix}
L^{\infty}(0,a) \\
L^{\infty}(0,b)
\end{pmatrix}$.

A computation shows that $V^{-1} \mathcal{F}^{-1} L^{\infty} (0, a)$ is bounded and lies in the model space $K_{u}^2$, where $u = e^{ia \frac{1-z}{1+z}}$ is an inner function on the disc. As $V$ and $\mathcal{F}$ are isomorphic and $L^{\infty}(0,a)$ is dense in $L^2(0,a)$, we must have $V^{-1} \mathcal{F}^{-1} L^{\infty} (0, a)$ is contained in $K_u^{\infty}$ and is dense subspace of $K_u^2$. A similar reasoning holds for $L^{\infty}(0,b)$. If we again assume that $G = \hat{H}$, where $H \mathcal{F}^{-1}\begin{pmatrix}
L^{\infty}(0,a) \\
L^{\infty}(0,b)
\end{pmatrix} \subseteq (L^2(\mathbb{R}))^2$ then we can deduce that $W_G$ is bounded if and only if $A_{H}^{\Theta}$ densely defined on $\mathcal{F}^{-1} \begin{pmatrix}
L^{\infty}(0,a) \\
L^{\infty}(0,b)
\end{pmatrix}$  is bounded if and only if the corresponding map on the disc, given by $A_{H  \circ m}^{\Theta \circ m } $ where $m: \mathbb{D} \to \mathbb{C}_+$ is defined by $m(z)=i\left(\frac{1-z}{1+z}\right)$ is bounded, i.e. if $A_{H  \circ m}^{\Theta \circ m } $ is an MTTO. Now with the above reasoning and Proposition \ref{boundedmttocondition}, whenever $G = \hat{H}$ with $H \mathcal{F}^{-1}\begin{pmatrix}
L^{\infty}(0,a) \\
L^{\infty}(0,b)
\end{pmatrix} \subseteq (L^2(\mathbb{R}))^2$ and $H  \circ m \in L^{( p, n \times n)}$ for $p \in ( 2, \infty ]$, we can test the continuity of \eqref{MTWH} by considering only the image of a corresponding modified MTTO.

In the case when $H \mathcal{F}^{-1}\begin{pmatrix}
L^{\infty}(0,a) \\
L^{\infty}(0,b)
\end{pmatrix} \subseteq (L^2(\mathbb{R}))^2$ and $W_G$ is bounded we have that $W_G$ is unitarily equivalent to $A_{H  \circ m}^{\Theta \circ m }  $ via the unitary mapping $ \mathcal{F} \circ V$.

Although the above demonstration is in the case of the 2-by-2 matrix $G$, it is easily generalised to the $n$-by-$n$ case. Matricial truncated Wiener-Hopf operators (which may also be called matricial convolution operators on finite intervals) are studied in detail in Chapter 7 of \cite{convolutionbook}. They are then shown to play an important role in Chapter 8 of \cite{convolutionbook}, where a continuous analogue of Krein's Theorem for matrix polynomials is given.

Matricial truncated Wiener-Hopf operators are also encountered naturally when finding the solution to MIMO (multi input, multi output) linear systems. Consider a matrix indicator function $\mathbbm{1}_{\mathbb{R}_-}$, where $\mathbbm{1}_{\mathbb{R}_-} (x) = I_{2 \times 2}$ if $ x \leqslant 0$ and $\mathbbm{1}_{\mathbb{R}_-} (x) = 0_{2 \times 2}$ if $ x > 0$, and set $G^'(x) = \exp (Ax) \mathbbm{1}_{\mathbb{R}_-}(x)  B$, where $A$ and $B$ are constant $2$-by-$2$ matrices. If we set $a=b$, then $ x \mapsto W_{G^'} (u) (x) + \exp (A x) v_{0}$ is the solution found for $v$ when solving the system of MIMO state space equations given by
\begin{align}
\dot{v} &=A v+B u, & v(0)=v_{0}, \\
y &=C v+D u, & &
\end{align}
where $u(x), v(x) ,y(x) \in \mathbb{C}^2$ are defined for all $x \in  (0,a)$, additionally $v$ is defined at 0 with the condition $v(0) = v_0$, $C$ and $D$ are constant $2$-by-$2$ matrices and $u \in (L^2(0, a))^2$. See Chapter 8 of \cite{JRPiis}, and in particular the solution of equation 8.1 for a more detailed discussion. The above working is easily generalised to the $n$-by-$n$ case. MIMO systems are encountered in control theory, dynamical systems, electrical engineering (see Chapter 5 of \cite{MIMOelectircal}) and find further applications in  wireless communication and multi-channel digital transmission.


\subsection*{Acknowledgments}
The author is grateful to the EPSRC for financial support. \newline
The author is grateful to Professor Partington for his valuable comments.

\subsection*{Declarations}

This research was supported by the EPSRC (grant number 1972662). \newline
\noindent Conflicts of interest/Competing interests- none.
\newline Availability of data and material- not applicable.
\newline Code availability- not applicable.
\newline Data sharing not applicable to this article as no datasets were generated or analysed during the current study.

\bibliographystyle{plain}
\bibliography{bibliography.bib}

\end{document}